\documentclass[12pt]{amsart}
\usepackage{amsmath,amssymb,amsbsy,amsfonts,latexsym,amsopn,amstext,cite,
                                               amsxtra,euscript,amscd,bm}
\usepackage{url}

\usepackage{mathrsfs}
\usepackage{color}
\usepackage[colorlinks,linkcolor=blue,anchorcolor=blue,citecolor=blue,backref=page]{hyperref}
\usepackage{breakurl}
\usepackage{cleveref}
\usepackage{color}
\usepackage[shortlabels]{enumitem} 
\usepackage{graphics,epsfig}
\usepackage{graphicx}
\usepackage{float}
\usepackage{epstopdf}
\usepackage[normalem]{ulem}
\hypersetup{breaklinks=true}

\usepackage[np]{numprint}
\npdecimalsign{\ensuremath{.}}

\usepackage{bibentry}

\usepackage[english]{babel}
\usepackage{mathtools}
\usepackage{todonotes}

\usepackage[norefs,nocites]{refcheck}

\def\ge{\geqslant}
\def\leq{\leqslant}
\def\geq{\geqslant}

\makeatletter
\newcommand{\refcheckize}[1]{%
  \expandafter\let\csname @@\string#1\endcsname#1%
  \expandafter\DeclareRobustCommand\csname relax\string#1\endcsname[1]{%
    \csname @@\string#1\endcsname{##1}\@for\@temp:=##1\do{\wrtusdrf{\@temp}\wrtusdrf{{\@temp}}}}%
  \expandafter\let\expandafter#1\csname relax\string#1\endcsname
}
\newcommand{\refcheckizetwo}[1]{%
  \expandafter\let\csname @@\string#1\endcsname#1%
  \expandafter\DeclareRobustCommand\csname relax\string#1\endcsname[2]{%
    \csname @@\string#1\endcsname{##1}{##2}\wrtusdrf{##1}\wrtusdrf{{##1}}\wrtusdrf{##2}\wrtusdrf{{##2}}}%
  \expandafter\let\expandafter#1\csname relax\string#1\endcsname
}
\makeatother

\refcheckize{\cref}
\refcheckize{\Cref}
\refcheckizetwo{\crefrange}
\refcheckizetwo{\Crefrange}
    
\newtheorem{theorem}{Theorem}
\newtheorem{lemma}[theorem]{Lemma}

\theoremstyle{remark}

\numberwithin{equation}{section}
\numberwithin{theorem}{section}
\numberwithin{table}{section}
\numberwithin{figure}{section}

 \def\({\left(}
 \def\){\right)}
 
 \def\mand{\qquad \mbox{and} \qquad}

\makeatletter
\def\paragraph{\@startsection{paragraph}{4}%
    \z@\z@{-\fontdimen2\font}%
{\normalfont\bfseries}}
\makeatother

\def\cM{{\mathcal M}}

\def\cS{{\mathcal S}}
\def\cT{{\mathcal T}}

\def\K{\mathbb{K}}
\def\Z{\mathbb{Z}}
\def\R{\mathbb{R}}
\def\Q{\mathbb{Q}}

\def\L{\mathbb{L}}

\def\fa{\mathfrak{a}}
\def\fb{\mathfrak{b}}

\def\fd{\mathfrak{d}}
\def\fh{\mathfrak{h}}
\def\fp{\mathfrak{p}}

\def\o{{\mathfrak o}}

\def\dD{D}

\def\ord{\mathrm{ord}}
\def\PrePer{\mathrm{PrePer}}

\def\MK{\cM}
\def\degf{\fd}
\def\sumloglogprimes{\mathscr L}
\def\underprime{\rho}
\def\allprimes{\mathscr P}
\def\suppset{\mathrm{supp}}
\def\largestsupp{\mathfrak{P}}

\allowdisplaybreaks

\definecolor{olive}{rgb}{0.3, 0.4, .1}
\definecolor{dgreen}{rgb}{0.,0.5,0.}

\definecolor{dgreen}{rgb}{0.,0.6,0.}

\DeclarePairedDelimiter\abs{\lvert}{\rvert}%
\DeclarePairedDelimiter\floor{\lfloor}{\rfloor}%

\DeclareMathOperator{\Nm}{N}

\def\yideal{\mathfrak{I}}
\def\cht{\hat{h}}
\def\loght{\lambda}
\def\sint{a}

\begin{document}

\title[ Multiplicatively Dependent Orbits Modulo $\cS$-Integers]{Effective Bounds on Multiplicatively Dependent Orbits of Integer Polynomials Modulo $\cS$-Integers}

\author[R. Li]{Ray Li}
\address{Department of Pure Mathematics, University of New South Wales,
Sydney, NSW 2052, Australia}
\email{rayli.main@gmail.com}

 \author[I. E. Shparlinski] {Igor E. Shparlinski}

\address{Department of Pure Mathematics, University of New South Wales,
Sydney, NSW 2052, Australia}
\email{igor.shparlinski@unsw.edu.au}

\date{}

\begin{abstract}
We obtain effective bounds on the heights of algebraic integers
whose orbits contain multiplicatively dependent values modulo $\cS$-integers.
Our method is based on a new upper bound on the so-called $\cS$-height of polynomial values 
over the ring of integers of $\K$.
Our results provide an effective variant of a recent result of  
A.~B\' erczes, A.~Ostafe, I.~E.~Shparlinski and 
J.~H.~Silverman (2019) on multiplicative dependence modulo a finitely generated subgroup
by eliminating the use of non-effective results by K.~F.~Roth and G.~Faltings. 
\end{abstract}


\maketitle

\section{Introduction}

\subsection{Background} 
For a polynomial  $f(X) \in \K[X]$ and $n \geq 0$, we write $f^{(n)}(X)$ for the $n$th iterate of $f$, that is, $f^{(0)}(X) = X$ and 
        \[
            f^{(n)}(X) = \underbrace{f \circ f \circ \ldots \circ f}_{n\:\text{times}}(X), \qquad n \ge 1.
        \]
The orbit of $\alpha \in \K$ is the set $\{\alpha, f(\alpha), f^{(2)}(\alpha), \ldots\}$.
In case the set is finite we say that $\alpha$ is  \textit{preperiodic\/} and we use
 $\PrePer(f)$ to denote the set of preperiodic points $\alpha \in \K$.

A famous theorem of Northcott~\cite{North}
says that for any number field  $\K$, for any
nontrivial polynomial
$f(X) \in \K[X]$ the set $\PrePer(f)$ is finite.
Namely there are only finitely many $\alpha \in\K$ such that
\begin{equation}
 \label{eq:Northcott} 
 f^{(m)}(\alpha) = f^{(n)}(\alpha)
\end{equation}
for two distinct iterations of $f$ (that is, for $m \neq n$).

Coupled with modern counting 
results on the number of algebraic numbers of bounded height and degree, see~\cite{Bar1,Bar2,Wid1,Wid2},
one can obtain various effective and rather explicit versions of this result. 

Several generalisations of the finiteness result of Northcott~\cite{North}
have recently been considered in~\cite{BOSS,OSSZ,OstShp}, 
where
\Cref{eq:Northcott}  has been replaced by 
various restrictions of multiplicative type on the ratios $f^{(m)}(\alpha) /f^{(n)}(\alpha)$ or even 
the ratios of the powers  $f^{(m)}(\alpha)^r /f^{(n)}(\alpha)^s$.

For example, it is shown in~\cite[Theorems~1.3 and~1.4]{BOSS} that if   $f(X) \in \K[X]$ of degree $d \ge 2$
is not of the form    $f(X)=  a X(X-b)^{d-1}$  with $a \in \K^*$ and $b \in \K$, then for any 
finitely generated multiplicative subgroup $\Gamma \in \K^*$, there are only finitely many 
$\alpha \in \K$ for which
\begin{equation}
\label{eq:Northcott-Gamma} 
f^{(n)}(\alpha) /f^{(m)}(\alpha)\in \Gamma
\end{equation}
for some integers $m > n \ge 0$.

\subsection{New results} 
Unfortunately the method  of~\cite{BOSS} relies on the results of Faltings~\cite{Falt1,Falt2} 
and thus is not effective.
We consider the more general equation
\begin{equation}
    \label{eq:mnuIntegers}
    f^{(n)}(\alpha) = \sint f^{(m)}(\alpha)
\end{equation}
for some integers $m > n \geq 0$ and an $\cS$-integer $\sint \in \o_\cS$
(see \Cref{eq:SIntDefn} for a definition).
Note that \Cref{eq:Northcott-Gamma} is a special case of \Cref{eq:mnuIntegers}.
However, \Cref{eq:mnuIntegers} is no longer symmetric in $m$ and $n$.

In \Cref{thm:MultDepNorthcottNF} (see also \Cref{thm:LowerBoundm}) we show that in the case where $f(X) \in \o[X]$, where $\o$ is the ring of integers of $\K$, 
and $\alpha \in \o$,
one can obtain an effective bound on the size of $\alpha$.
In fact, we trace the explicit dependence on $\cS$.
This is an effective version of~\cite[Theorem~1.4]{BOSS}.

Furthermore, we also provide an effective variant of~\cite[Theorem~1.7]{BOSS}
which states that, under mild additional constraints, there are only finitely many $\alpha \in \K$ that
satisfy the following relation of multiplicative dependence modulo $\cS$-units among values in an orbit
\[
    f^{(n+k)}(\alpha)^r \cdot f^{(k)}(\alpha)^s \in \o_\cS^*
\]
for some $n,k \geq 1$ and $(r,s) \neq (0,0)$.
That is, we give an  \textit{effective upper bound} on the height
of $\alpha \in \o$
that satisfy
\begin{equation}
    \label{eq:MultDepRelation}
    \left(f^{(m)}(\alpha)\right)^r = u \left(f^{(n)}(\alpha)\right)^s
\end{equation}
for some integers  $m > n \geq 1$,
$(r,s) \neq (0,0)$ and
an  $\cS$-unit $u \in \o_\cS^*$,  see \Cref{eq:S-unit} for a 
definition.

As in~\cite{BOSS}, the key to proving \Cref{thm:MultDepNorthcottNF}
is an upper bound on the  \textit{$\cS$-height of polynomial values\/}, see \Cref{eq:S-height} for a 
precise definition, which we believe 
is of independent interest and may find other applications.
Recall that in~\cite{BOSS} this upper bound is provided by~\cite[Theorem~11(c)]{HsiaSilverman}
which is unfortunately not effective.
Here we modify the argument of~\cite{BOSS}
to use an effective variant of~\cite[Theorem~11(c)]{HsiaSilverman}
which we provide by
extending~\cite[Theorem~2.2]{BEG} to number fields. 

We note that obtaining such extensions 
in terms of the norm in $\o$ can be done by following the arguments in~\cite{BEG}.
However, such a generalisation is not sufficient for our purpose, so we add some additional 
ideas and ingredients in order to obtain a bound in terms of the $\cS$-height (see \Cref{eq:S-height}).

\subsection{General notation and conventions} 
\label {sec:note}

We set the following notation which we use for the rest of this paper. 
We refer to~\cite{BomGub} for a background on valuations, height and other notions 
we introduce below.

Throughout this paper, we assume that $\K$ is a number field of degree $d$, with class number $\fh$, regulator $R$ and ring of integers  $\o$.

We use $\MK$ to denote the set of places of $\K$ and write 
\[
\MK =   \MK^{\infty} \cup \MK^0, 
\]
where  $\MK^{\infty}$  and $\MK^0$ are the set of archimedean (infinite) and non-archimedean (finite) places of $\K$ respectively.

We always assume that $\cS$  is  a finite set of  places containing $\MK^{\infty}$
and use $\cS_0 = \cS \cap \MK^0$ to  denote the set of finite places in $\cS$. We also define
\[
s = \# \cS \mand  t = \# \cS_0.
\]

As usual, $\o_\cS^*$ denotes the group of $\cS$-units, that is
\begin{equation}
  \label{eq:S-unit}
  \o_\cS^*= \{u \in \K^*:~\abs{u}_v = 1 \ \forall v \in \cM \setminus \cS \}. 
\end{equation}
In particular, $\o^* =  \o_{\MK^{\infty}}^*$ is the group of units which, by the Dirichlet Unit Theorem, is
a finitely generated group of rank $\#  \MK^{\infty} -1$.

Similarly, $\o_\cS$ denotes the ring of $\cS$-integers, that is
\begin{equation}
  \label{eq:SIntDefn}
  \o_\cS= \{a \in \K:~\abs{a}_v \leq 1 \ \forall v \in \cM \setminus \cS \}. 
\end{equation}

We use $\Nm(\fa)$ for the norm of the ideal $\fa$, we also write  
$\Nm(\alpha)$ to mean  $\Nm([\alpha])$, where $[\alpha]$ is the principal ideal in $\o$ generated by $\alpha \in \o$.
In particular, $\Nm(\alpha) >0$ for $\alpha \ne 0$. 

For $x \geq 0$
it is convenient to introduce the functions 
\[
\log^+ x = \max\{\log x, 0\} \mand \log^*x = \max\{\log x, 1\}, 
\]
with $\log^+ 0 = 0$, $\log^* 0 = 1$.
We are now able to define the   \textit{logarithmic height\/} of  $\alpha  \in \K$ as 
\[
h(\alpha) = \sum_{v \in \MK} \frac{\ell_v}{d} \log^+\abs{\alpha}_v, 
\]  
where
\begin{itemize}
    \item $\abs{\alpha}_v$ is the absolute value extending
        the valuation on $\Q$.
        That is, for a finite place $v$ corresponding to a prime ideal $\mathbf{p} \mid p$
        \[
            \abs{\alpha}_v = p^{-\ord_\mathbf{p}\, \alpha/e_v},
        \]
        where $\ord_\mathbf{p}\, \alpha$ is the $\mathbf{p}$-adic order of $\alpha$.
\item  $\ell_v$ denotes the local degree of the  valuation $v$, that is
        \[
            \ell_v = [\K_v : \Q_v], 
        \]
        where $\K_v$ and $\Q_v$ are the completions at $v$.
\end{itemize}

Finally, for a set $\cT \subseteq \MK$ we use  
\begin{equation}
\label{eq:S-height}
            h_\cT(\alpha) =  \sum_{v \in \cT} \frac{\ell_v}{d} \log^+\abs{\alpha}_v
\end{equation}
to denote the  \textit{$\cT$-height} of  $\alpha  \in \K$. 

We also recall the identity, 
\begin{equation}
\label{eq:sum l/d}
\sum_{v_i \in \MK^\infty} \frac{\ell_{v_i}}{d} = 1.
\end{equation}
which is a special case of~\cite[Corollary~1.3.2]{BomGub} applied to the archimedean valuation of $\Q$.
 
Let $\allprimes_\K$ denote the set of all prime ideals of $\o$.
For $\alpha \in \K^*$ define
\[
    \suppset(\alpha) = \{\mathbf{p} \in \allprimes_\K \mid \ord_\mathbf{p}(\alpha) > 0\}
\]
and
\[
    \largestsupp(\alpha) = \max_{\mathbf{p} \in \suppset(\alpha)} \Nm(\mathbf{p})
\]
with the convention that $\largestsupp(\alpha) = 1$ if $\suppset(\alpha) = \varnothing$.

We use $A$ with or without 
subscripts or arguments  for fully  \textit{explicit} constants, while $c$ and $C$
are used for not explicit but   \textit{effective\/}
constants depending on their arguments.

\section{Main Results}

\subsection{Height of \texorpdfstring{$\cS$}{S}-parts of polynomials in number fields}
We start with an 
effective version 
of~\cite[Theorem~1.4]{BOSS}
where we also make the dependence on $\cS$ completely explicit. 
This result is proven by
extending~\cite[Theorem~2.2]{BEG} to number fields.  We note that it is 
also indicated in~\cite{BEG} that such an extension to number fields should be possible.
However, if one follows closely the argument of the proof of~\cite[Theorem~2.2]{BEG}  
this leads to such an extension in terms of the norm, while for our purpose we need 
it in terms of the height, which requires bringing in additional tools.

First we need to define some notation stemming from the use
of~\cite{GyoryYu}.

Suppose we are working with a finite set of
places $\cS$ and the prime ideals
$\{\mathbf{p}_1, \cdots, \mathbf{p}_t\}$
correspond to the places in $\cS_0$.
Then define
\begin{equation}
    \label{eq:PQTdefn}
    P = \max\limits_{i \in [1,t]} \Nm(\mathbf{p_i}),\quad Q = \Nm(\mathbf{p_1\cdots p_t}),
    \quad \sumloglogprimes = \sum_{i=1}^t \log^* \log \Nm(\mathbf{p_i})
\end{equation}
with the convention that  for  $\cS = M^{\infty}$ we set $P = Q = 1$, $\sumloglogprimes = 0$.   

Also define the functions
\begin{equation}
    \label{eq:A1defn}
    A_1(u,v) = v^{2v + 3.5}2^{7v}\log(2v)u^{2v}
\end{equation} 
and
\begin{equation}
    \label{eq:A2defn}
    A_2(u,v) = (2048u)^v v^{3.5}
\end{equation}
which stem from~\cite{GyoryYu} which underlies our argument.

We recall the definition of $h_\cS$ in Section~\ref{sec:note} and also that $d = [\K:\Q]$. 

\begin{theorem}
    \label{thm:SPartBoundNF}
    Let $f(X) \in \o[X]$ be a polynomial with at least 3 distinct roots.
    Let $\L$ be a splitting field of $f$ over $\K$, let $\dD = [\L : \K]$ and let $\fh_\L$ denote the class number of $\L$.
    Let $\cS$ be a finite set of $s$ places of $\K$ containing all infinite places
    and let $t = \#\cS_0$.
    Then for all $\alpha \in \o$, $f(\alpha) \neq 0$
    we have
    \begin{equation}
        \label{eq:SPartBoundGeneric}
        h_\cS(f(\alpha)^{-1})
        < (1-\eta_1(\K,f,\cS))
        (h(f(\alpha)) + 1),
    \end{equation}
    where
    \begin{align*}
    \eta_1(\K,f,\cS)^{-1}  & = c_1(\K, f) A_1(d \dD, s\dD) \max\{1,t\}\\
    & \qquad \qquad  \times P^{\dD}  (\log^* P+ \sumloglogprimes)\prod_{i=1}^t \log (\Nm(\mathbf{p_i}))^{\dD},
    \end{align*} 
    and, for $t > 0$, we have
    \begin{equation}
        \label{eq:SPartBound2}
        h_\cS(f(\alpha)^{-1})
        < (1-\eta_2(\K,f,\cS))
        (h(f(\alpha)) + 1),
    \end{equation}
    where
    \[
    \eta_2(\K,f,\cS)^{-1}   = c_1(\K, f) A_2(d \dD \fh_\L, t\dD) t      P^{\dD} \prod_{i=1}^t \log (\Nm(\mathbf{p_i}))^{\dD},
    \]
    where $c_1(\K,f) > 0$ is effectively computable.
\end{theorem}

Note that \Cref{eq:SPartBound2}
omits the $v^v$ term
in~\Cref{eq:A1defn}.
This is necessary for the proofs of
\Cref{thm:LowerBoundm,thm:WeakZsigmondy}.

We also note that, using
the recent improvement~\cite[Corollary 4]{Gyory}
in place of~\Cref{lem:BoundhDecomposableForm},
we can replace the main dependence on $P$
by a dependence on the third largest value of
$\Nm(\mathbf{p_i})$, $i = 1, \ldots, t$.

\subsection{Effective bounds on points with multiplicatively dependent orbits} 
We recall that $d = [\K : \Q]$.

\begin{theorem}
    \label{thm:MultDepNorthcottNF}
    Let $f(X)\in\o[X]$ be a polynomial with at least 3 distinct roots
    and for which $0$ is not periodic.
    Let $\cS$ be a finite set of places of $\K$ containing all infinite places
    and let $t = \# \cS_0$.
    Then for any $\alpha \in \o$ such that \Cref{eq:mnuIntegers} holds 
    for some non-negative integers $m > n$ and $\sint \in\o_{\cS}$ we have
    \begin{equation}
        \label{eq:MultDepNorthcottBound1}
        \begin{split}
            h(\alpha) <\ & c_2(\K, f) \eta_1(\K, f, \cS)^{-1},
        \end{split}
    \end{equation}
    and, for $t > 0$,
    \begin{equation}
        \label{eq:MultDepNorthcottBound2}
    \begin{split}
        h(\alpha) <\ & c_2(\K, f) \eta_2(\K,f,\cS)^{-1},
    \end{split}
    \end{equation}
    where $\eta_1(\K,f,\cS)$, $\eta_2(\K,f,\cS)$
    are as in \Cref{thm:SPartBoundNF}
    and
    $c_2(\K, f)$ is an effectively computable constant.
\end{theorem}

With \Cref{thm:MultDepNorthcottNF} we can also prove the following effective variant
of \cite[Theorem~1.7]{BOSS}.

\begin{theorem}
    \label{thm:MultDepGeneralThm}
    Let $f(X) \in \o[X]$ be a polynomial of degree at least 3 without multiple roots and for which 0 is not periodic.
    Let $\cS$ be a finite set of places of $\K$ containing all infinite places.
    Then for any tuple $(m,n,\alpha,r,s)$ for which \Cref{eq:MultDepRelation} holds with $n \geq 1$
    we have
    \[
        h(\alpha) < c_3(\K, f, \cS)
    \]
    for some effectively computable $c_3(\K,f,\cS)$.
\end{theorem}

Note that we have assumed $m,n \neq 0$, otherwise there are trivially
infinitely many solutions of the form
$\left(f^{(m)}(u)\right)^0 = u^{-1} \left(f^{(0)}(u)\right)$.

\Cref{thm:MultDepGeneralThm} almost directly follows from the proof of \cite[Theorem~1.7]{BOSS}
but instead using \Cref{thm:MultDepNorthcottNF} in place of \cite[Theorem~1.3]{BOSS}.

\subsection{Applications to the existence of large prime ideals in factorisations}

For $\alpha \in \K$, define the function
\[
    \loght(\alpha) = \log^* h(\alpha).
\]

We obtain an effective lower bound on the largest norm of a prime ideal
appearing with a higher order in $f^{(m)}(\alpha)$ than in $f^{(n)}(\alpha)$.

\begin{theorem}
    \label{thm:LowerBoundm}
    Let $f(X)\in\o[X]$ be a polynomial with at least 3 distinct roots
    and for which 0 is not periodic.
    Let $\alpha \in \o$, $m, n \in \Z$, $m > n \geq 0$ such that $f^{(m)}(\alpha)$, $f^{(n)}(\alpha) \neq 0$.
    Then
    \[
        \largestsupp\left(\frac{f^{(m)}(\alpha)}{f^{(n)}(\alpha)}\right)
        > c_4(\K,f)
        \frac{\loght\left(f^{(m)}(\alpha)\right) \log^* \loght\left(f^{(m)}(\alpha)\right)}{\log^* \log^* \loght\left(f^{(m)}(\alpha)\right)},
    \]
    where $c_4(\K, f) > 0$ is an effectively computable constant.
\end{theorem}

Using standard properties of height, such as \Cref{eq:hfnalphaIneqs} below, we see that  
\Cref{thm:LowerBoundm} implies that
if, in addition, $\alpha$ is not preperiodic, 
then 
  \[
        \largestsupp\left(\frac{f^{(m)}(\alpha)}{f^{(n)}(\alpha)}\right)
        >    c_5(\K,f)
        \frac{m  \log^*  m}{\log^* \log^* m},
\]
where $c_5(\K,f) > 0$ is an effectively computable constant.

Finally, we obtain a result on the existence of primitive divisors
within small sets of iterates.

\begin{theorem}
    \label{thm:WeakZsigmondy}
    Let $f(X) \in \o[X]$ be a polynomial with at least 3 distinct roots
    and for which 0 is not periodic.
    Then there exists an effectively computable constant $c_6(\K, f) > 0$ such that, letting
    \[
        k(m, \alpha) = \floor{c_6(\K, f) \log \lambda(f^{(m)}(\alpha))},
    \]
    for every $m \in \Z$, $m > 0$, and every $\alpha \in \o$,
    $f^{(m)}(\alpha)$ not a unit,
    there exists a prime ideal $\mathbf{p}$ that divides
    $f^{(m)}(\alpha)$ but does not divide any element in the set
    \[
        \{ f^{\left(\max(0,m-k(m,\alpha))\right)}(\alpha), f^{\left(\max(0,m-k(m,\alpha))+1\right)}(\alpha), \cdots, f^{(m-1)}(\alpha)\}.
    \]
\end{theorem}

If, in addition, $\alpha$ is not preperiodic, then,
using \Cref{eq:hfnalphaIneqs}, \Cref{thm:WeakZsigmondy} also holds for
\[
    k(m, \alpha) = \floor{c_7(\K, f) \log m}
\]
for some effectively computable $c_7(\K, f) > 0$.

\section{Proof of \texorpdfstring{\Cref{thm:SPartBoundNF}}{Theorem~\ref{thm:SPartBoundNF}}}

\subsection{Preliminaries} 
As in the proof of~\cite[Theorem~2.10]{BEG}, the main tool
is~\cite[Theorem~3]{GyoryYu}.
We state the special case for 2 variables where it is
easy to state a sufficient condition for $F$ to be \textit{triangularly connected}.
We maintain the dependence on $\cS$; however, we omit the
explicit dependence on $\K$ and $F$.

Let $R$ be the regulator of $\K$.
    
In  \Cref{lem:BoundhDecomposableForm}  below, which is a simplified version of~\cite[Theorem~3]{GyoryYu},
 we have made use of the inequality
(see~\cite{BugeaudGyory}) 
\[
    R_\cS \leq \fh R \prod_{i=1}^{t} \log \Nm(\mathbf{p_i}), 
\]
where $\fh$ is the class number of $\K$ and   $R_\cS$ is the $\cS$-regulator of $\K$ 
(see~\cite{BugeaudGyory} for a definition,  it is the natural generalisation of the regulator to S-units).
In particular, we absorb $\fh, R$ into the constant $C_1(\K,F)$.

We recall that a binary form (that is, a homogeneous polynomial) 
$F \in \K[X,Y]$  is called   \textit{decomposable  over $\K$\/}, if 
$F$ factors into linear factors over $\K$.

\begin{lemma}
    \label{lem:BoundhDecomposableForm}
    Let $F \in \K[X,Y]$ be a decomposable form over $\K$ which has at least 3
    pairwise non-proportional linear factors.
    Let $\cS$ be a finite set of $s$ places of $\K$ containing all infinite places
    and let $t = \# \cS_0$.
    Let $P, Q, \sumloglogprimes$ be as defined in
    \Cref{eq:PQTdefn}.    Let $\beta \in \K \setminus \{0\}$.
    Then all solutions $(x_1, x_2) \in \o_\cS^2$ of
    \[
        F(x_1, x_2) = \beta
    \]
    satisfy
    \begin{equation}
        \label{eq:DecomposableBound1}
        \begin{split}
        h(x_1), h(x_2) &<
        C_1(\K,F)A_1(d,s)  \left(\log^* Q + h(\beta) \right)   \\
                       &\qquad\qquad\qquad \times 
        P  (1+ \sumloglogprimes/\log^* P)\prod_{i=1}^t \log \Nm(\mathbf{p_i})                     , 
     \end{split}
 \end{equation}
     and, for $t > 0$,
     \begin{equation}
        \label{eq:DecomposableBound2}
    \begin{split}
        h(x_1), h(x_2) &<
        C_2(\K,F)A_2(d\fh,t)  \left(\log^* Q + h(\beta) \right)   \\
                       &\qquad\qquad\qquad \times 
                       (P/\log^* P)\prod_{i=1}^t \log \Nm(\mathbf{p_i}),
     \end{split}
     \end{equation}
     where $A_1$ is defined as in \Cref{eq:A1defn}, $A_2$ is defined as in \Cref{eq:A2defn},
     $d = [\K : \Q]$, $\fh$ is the class number of $\K$ and  $C_1(\K,F)$, $C_2(\K,F)$
     are effectively computable constants.
\end{lemma}

We refer to~\cite{GyoryYu} for a fully explicit statement.
For the case $t > 0$
see also the recent improvement~\cite[Corollary 4]{Gyory}.

To adapt the proof of~\cite[Theorem~2.10]{BEG} to number fields,
we need a well known fact on the approximation of archimedean valuations by units.
To obtain explicit bounds, we first need~\cite[Lemma~1]{BugeaudGyory}
in the case where $\cS = \MK^{\infty}$
(see also~\cite[Lemma~2]{GyoryYu} for an alternative bound when the unit rank is at least 2).

\begin{lemma}
    \label{lem:UnitSystemBounds}
    Let $\K$ be a number field with unit rank at least 1.
    Then there exists a fundamental system of units
$\varepsilon_1, \ldots, \varepsilon_{r}$
such that
 \[
 \max\limits_{1 \leq i \leq r} h(\varepsilon_i) \leq A_3(\K) R,\]
 where
\[
    A_3(\K) = \frac{(r!)^2}{2^{r-1}d^{r}}\left(\frac{\delta_\K}{d}\right)^{1-r},
\]
where $d = [\K:\Q]$
and $\delta_\K$ is any positive constant such that every non-zero algebraic number $\alpha \in \K$
which is not a root of unity
satisfies $h(\alpha) \geq \delta_\K/d$.
\end{lemma}

A result of Voutier~\cite{Voutier} states that we can take 
\[
    \delta_\K = \begin{cases}
    \displaystyle{\frac{\log 2}{d}}  &\text{if $d = 1, 2$}, \\
       \displaystyle{ \frac{1}{4}\left(\frac{\log \log d}{\log d}\right)^3 } &\text{if $d \geq 3$}, 
    \end{cases}
\]
in \Cref{lem:UnitSystemBounds}.

In the following result, little effort has been made to optimise the
right hand side as it suffices for our results that it is effectively computable
in terms of $\K$. In fact, it is essentially established in the proof of~\cite[Lemma~2]{BugeaudGyory}
(see also~\cite[Lemma~3]{GyoryYu}); however, for the sake of completeness,  we give a short proof. 

\begin{lemma}
    \label{lem:UnitExistsHtNm}
    For every $\alpha \in \o \setminus \{0\}$ and for every integer $n \geq 1$
    there exists an $\varepsilon \in \o^*$ such that
    \begin{equation}
        \label{eq:UnitHtNm}
        \abs*{\log \abs{\varepsilon^n \alpha}_{v} -
        \frac{1}{d} \log(\Nm(\alpha))} \leq \frac{1}{2}A_3(\K)nd^2R
    \end{equation}
    for all $v\in \MK^{\infty}$ with $A_3(\K)$ as in \Cref{lem:UnitSystemBounds}
    and $d = [\K : \Q]$.
\end{lemma}

\begin{proof}
    For this proof let $\MK^\infty  = \{v_1, \ldots, v_{r+1}\}$.
    Since the case $r = 0$ is trivial, we henceforth assume that $r \geq 1$.
    Let $\varepsilon_1, \ldots, \varepsilon_{r}$ be a fundamental system of units satisfying
    the inequalities of  \Cref{lem:UnitSystemBounds}.

    We note that by the Dirichlet Unit Theorem (see, for example,~\cite[Theorem~I.7.3]{Neukirch}),
    the columns of the $(r + 1) \times r$ matrix $M$ with
    \[
        M_{i,j} = \ell_{v_i} \log \abs{\varepsilon_j}_{v_i}
    \]
    (where as before $\ell_v = [\K_v : \Q_v]$ for $v\in \MK$) 
    form a basis for the hyperplane in $\R^{r+1}$
    of vectors whose coordinates sum to 0.

    Let $\mathbf{v}$ be the
    column vector of dimension $r+1$, where
    \[
        (\mathbf{v})_i = \ell_{v_i} \log\(\Nm(\alpha)^{-1/d}\abs{\alpha}_{v_i}\), \qquad i =1, \ldots, r+1.
    \]
    Then there exists a unique vector $\mathbf{x} = (x_1, \ldots, x_{r})^T$ such that
    \[
        M \mathbf{x} = \mathbf{v}.
    \]
    For each $i=1, \ldots, r$, we write
    \[
        x_i = ny_i + z_i,\quad y_i, z_i \in \Z, \ z_i \in \left(-\frac{n}{2}, \frac{n}{2}\right].
    \]

    Let $\varepsilon = \varepsilon_1^{-y_1}\cdots\varepsilon_{r}^{-y_{r}}$.
    Then, for all $v_i$,
    \[
        \log \abs{\varepsilon_1^{z_1} \cdots \varepsilon_{r}^{z_{r}}}_{v_i}
        =
        \log \abs{\varepsilon^n\alpha}_{v_i} -
        \frac{1}{d}\log(\Nm(\alpha)).
    \]
    For each $j=1, \ldots, r$, by  \Cref{lem:UnitSystemBounds}, we have
    \[
        \abs{z_j \log \abs{\varepsilon_j}_{v_i}}
        \leq \frac{n}{2} d h(\varepsilon_j)
        \leq \frac{n}{2} d A_3(\K) R
    \]
 and summing over $j=1, \ldots, r$ yields the desired statement.
\end{proof}

We now have all the tools we need for the proof of
\Cref{thm:SPartBoundNF}.

\subsection{Concluding the proof} 
    We will first prove \Cref{eq:SPartBoundGeneric} holds
    assuming that $f$ splits in $\K$.

    Let $F(X,Y)$ be the homogenisation of $f$, that is, 
\[
  F(X,Y) = Y^\degf f(X/Y),
\] 
where     $\degf = \deg f$. 
    Then $F$ is a decomposable form in $\K$ with $F(x,1) = f(x)$.

    Suppose $x \in \o$ and  $f(x) \neq 0$.
    Let
    $b = F(x,1) = f(x)$.
    We can write $[b]$ uniquely in the form
    \begin{equation}
        \label{eq:bIdealDecomp}
        [b] = \mathbf{p_1}^{b_1}\ldots\mathbf{p_t}^{b_t}\mathbf{a},
    \end{equation}
    where $\mathbf{a}$ is an ideal coprime to $\mathbf{p_1},\ldots,\mathbf{p_t}$
    and $b_i = \ord_{\mathbf{p_i}}\,  b$, $i=1, \ldots, t$.

    Decompose each $b_i$ (uniquely) as
    \[
        b_i = \degf\fh q_i+ r_i,
    \]
    where $\fh$ is the class number of $\K$
    and $q_i, r_i \in \Z_{\geq 0}$, $r_i < \degf\fh$.

    For each $i \in [1, t]$, define $p_i \in \o$
    to be any generator of $\mathbf{p_i}^{\fh}$
    (which is a principal ideal).

    Now let
    \begin{equation}
        \label{eq:Defc}
        c = F\left(\frac{x}{p_1^{q_1}\ldots p_t^{q_t}}, \frac{1}{p_1^{q_1}\ldots p_t^{q_t}}\right)
        = \frac{b}{p_1^{\degf q_1} \ldots p_t^{\degf q_t}},
    \end{equation}
    so that
    \begin{equation}
        \label{eq:cideal}
        [c] = \mathbf{p_1}^{r_1}\ldots\mathbf{p_t}^{r_t}\mathbf{a}.
    \end{equation}

We now apply \Cref{lem:UnitExistsHtNm} to $c$ and
    let $\varepsilon \in \o^*$ be any unit
    satisfying \Cref{eq:UnitHtNm} where
    $\alpha$ is replaced by $c$
    and $n $ by $\degf$.

    Multiplying the arguments of $F$ by $\varepsilon$ we get
    \[
        F\left(\frac{\varepsilon x}{p_1^{q_1}\ldots p_t^{q_t}},
        \frac{\varepsilon}{p_1^{q_1}\ldots p_t^{q_t}}\right)
        = \varepsilon^\degf c
    \]
    and applying
    \Cref{lem:BoundhDecomposableForm}
    we obtain the inequality
\begin{equation} 
\begin{split}
        \label{eq:hLHSRHSineq}
        h(\varepsilon/p_1^{q_1}\ldots p_t^{q_t}) &<
        C_1(\K,f)A_1(d,s) \left(\log^* Q + h(\varepsilon^\degf c) \right)   \\
                                                   & \qquad \qquad
                                                   \times P (1+ \sumloglogprimes/\log^* P)\prod_{i=1}^t \log \Nm(\mathbf{p_i}).
\end{split}
\end{equation}

    We separately lower bound $h(\varepsilon/p_1^{q_1}\ldots p_t^{q_t})$
    and upper bound $h(\varepsilon^\degf c)$.

    \subsubsection*{--- Lower bound on $h(\varepsilon/p_1^{q_1}\ldots p_t^{q_t})$:}\quad 
    Since $\varepsilon/p_1^{q_1}\ldots p_t^{q_t}$ is an $\cS$-integer
    \begin{equation} 
\begin{split}
        \label{eq:hvarepsilon}
        h(\varepsilon/p_1^{q_1}\ldots p_t^{q_t}) =
                                                   &\sum_{\substack{v_i \in \MK^0 \cap \cS}}
                                                   \frac{\ell_{v_i}}{d} \log^+\abs{\varepsilon/p_1^{q_1}\ldots p_t^{q_t}}_{v_i} \\
                                                   &+ \sum_{v_i \in \MK^\infty} \frac{\ell_{v_i}}{d} \log^+\abs{\varepsilon/p_1^{q_1}\ldots p_t^{q_t}}_{v_i}.
\end{split}
\end{equation}
    From~\Cref{eq:Defc},
    we have that for all $v_i$
\[
\degf\log \abs{\varepsilon/p_1^{q_1}\ldots p_t^{q_t}}_{v_i}
        - \log \left(\abs{\varepsilon^\degf c}_{v_i}\right)
        = \log \left(\abs{b}_{v_i}^{-1}\right). 
 \]

    If $v_i \in \MK^0 \cap \cS$, by direct calculation we get
    \begin{equation}
        \label{eq:viFiniteLowerBound}
        \log^+ \abs{\varepsilon/p_1^{q_1}\ldots p_t^{q_t}}_{v_i}
        > \frac{1}{\degf}\log^+ (\abs{b}_{v_i}^{-1})
        - \frac{\fh}{e_{v_i}} \log \underprime_i
\end{equation}
    (where $\underprime_i$ is the prime in $\Z$ that $\mathbf{p_i}$ lies over
    and $e_{v_i}$ is the ramification index of $v_i$, $i =1, \ldots, t$.).

    If $v_i \in \MK^{\infty}$, using the bound of \Cref{lem:UnitExistsHtNm} and dividing by $\degf$ we get
\[
        \log \abs{\varepsilon/p_1^{q_1}\ldots p_t^{q_t}}_{v_i} \geq
        \frac{1}{\degf}\left(\log \left(\abs{b}_{v_i}^{-1}\right) + \frac{1}{d} \log(\Nm(c))\right) - \frac{1}{2} A_3(\K) d^2R. 
\]
    Hence
    \begin{equation} 
\begin{split}
            \label{eq:viInfiniteLowerBound}
        \log^+ \abs{\varepsilon/p_1^{q_1}\ldots p_t^{q_t}}_{v_i}
        &\geq
        \frac{1}{\degf}\log^+ \left(\abs{b}_{v_i}^{-1}\right) -  \frac{1}{2} A_3(\K) d^2R.
        \end{split}
\end{equation}
    
    Substituting \Cref{eq:viFiniteLowerBound} and \Cref{eq:viInfiniteLowerBound}
    into \Cref{eq:hvarepsilon} and using the trivial bound $ e_{v_i} \ge 1$
    and~\Cref{eq:sum l/d}  we obtain
    \begin{equation}
        \label{eq:LHSLowerBoundOnh}
        h(\varepsilon/p_1^{q_1}\ldots p_t^{q_t}) \geq
        \frac{1}{\degf} h_\cS(b^{-1})
        - \fh \log Q -   \frac{1}{2} A_3(\K) d^2R, 
    \end{equation}  
    where again we let 
    \[
    Q = \Nm(\mathbf{p_1}\cdots \mathbf{p_t}) \ge \underprime_1 \ldots \underprime_t.
    \] 

    \subsubsection*{--- Upper bound on $h(\varepsilon^\degf c)$:} \quad 
    Since $\varepsilon^\degf c \in \o$ we have
    \[
        h(\varepsilon^\degf c) = \sum_{v_i \in \MK^\infty} \frac{\ell_{v_i}}{d} \log^+ \abs{\varepsilon^\degf c}_{v_i}.
    \]
    From \Cref{eq:sum l/d} and \Cref{eq:UnitHtNm}
        we obtain
    \[
        h(\varepsilon^\degf c) \leq \frac{1}{d} \log (\Nm(c)) + \frac{1}{2} A_3(\K) \degf d^2 R.
    \]
    Since $r_i < \degf\fh$, from \Cref{eq:cideal} we get that
    \begin{equation}
        \label{eq:Proofhenc1dlog}
        h(\varepsilon^\degf c) \leq \frac{1}{d} \log (\Nm(\mathbf{a}))
        + \frac{\degf\fh}{d} \log Q
        + \frac{1}{2} A_3(\K) \degf d^2 R,
    \end{equation}
    where again we let $Q = \Nm(\mathbf{p_1}\cdots \mathbf{p_t})$.

    By definition $\ord_{\mathbf{p_i}}\, \mathbf{a} = 0$ 
    for all finite valuations in $\cS$.
    Substituting \Cref{eq:bIdealDecomp} into \Cref{eq:Proofhenc1dlog}
    we obtain the upper bound
    \begin{equation} 
\begin{split}
        \label{eq:RHSUpperBoundOnh}
        h(\varepsilon^\degf c) &\leq h_{\MK \setminus \cS}(b^{-1})
        + \frac{\degf\fh}{d} \log Q
        + \frac{1}{2} A_3(\K) \degf d^2 R.
\end{split}
\end{equation}

    \subsubsection*{--- Combining the bounds:} \quad 
    Substituting \Cref{eq:LHSLowerBoundOnh} and \Cref{eq:RHSUpperBoundOnh}
    into \Cref{eq:hLHSRHSineq} we obtain
    \begin{align*}
             h_\cS(b^{-1}) < C_3(\K, f)A_1(d,s) &
             \left( 
             \log^* Q +
             h_{\MK \setminus \cS}(b^{-1}) \right) \\
             &  \times
             P  (1+ \sumloglogprimes/\log^* P)\prod_{i=1}^t \log \Nm(\mathbf{p_i}), 
    \end{align*}
    where $C_3(\K, f)$  is an effectively computable constant.

    Noting that $\log Q \leq t \log P$
    we can simplify to get
    \[
         h_\cS(b^{-1})< C_3(\K,f) A_4(d,\cS)
         \left(1
         +
         h_{\MK \setminus \cS}(b^{-1})
         \right),
     \]
    where
    \[
             A_4(d,\cS) =  A_1(d,s) \max\{1,t\}  P  (\log^* P+ \sumloglogprimes)\prod_{i=1}^t \log \Nm(\mathbf{p_i}). 
    \]
Using
    \[
        h(b) = h(b^{-1}) = h_{\MK \setminus \cS}(b^{-1}) + h_\cS(b^{-1})
    \]
    we now arrive to 
    \begin{equation}
        \label{eq:hSbFinalBound}
         h_\cS(b^{-1})
         <
         \frac{ C_3(\K,f)  A_4(d,\cS)}
         {1+  C_3(\K,f)  A_4(d,\cS)}
         \left(1
         +
         h(b)
         \right), 
    \end{equation}
    concluding the proof of \Cref{eq:SPartBoundGeneric}
    in the case where $f$ splits in $\K$.

    \subsubsection*{--- Proving \Cref{eq:SPartBoundGeneric}:} \quad
    Now, suppose that $f$ does not split in $\K$.
    Let $\L$ be the splitting field of $f$ over $\K$
    and let $\cT$ be the set of places in $\L$ lying over $\cS$.

    Then \Cref{eq:hSbFinalBound} holds in $\L$
    where we replace $\cS$ by $\cT$.
    For ease of notation, we  introduce the subscript $\L$  
    when talking about constants defined in terms of $\L$ (some of them also depend on $\cS$).
    In particular,  $d_\L =   [\L : \Q]$, $s_\L = \# \cT$ and so on.     
    Let $\dD = [\L : \K]$.

We note that 
    \begin{equation}
    \begin{split}
        \label{eq:NewParam-L}
& \qquad d_\L = \dD d, \quad t_{\L} \leq  \dD t, \quad s_{\L} \leq \dD s, \\ 
& \quad P_\L \leq P^{\dD}, \qquad 
 \sumloglogprimes_\L \leq \dD \sumloglogprimes + C_4(\K,f), \\   
 &    \prod_{\mathbf{q} \in \cT_0} \log \Nm(\mathbf{q}) <
        C_5(\K, f) \prod_{\mathbf{p} \in \cS_0} (\log \Nm(\mathbf{p}))^{\dD},
    \end{split}
    \end{equation}
where $C_4(\K,f), C_5(\K,f)$ are effective constants that depend on $\dD$ and the number of prime ideals of $\K$ with norm less than $e^e$.

    We also note that heights are independent of extension
    in the sense that if $b \in \K$, then
    \[
        h(b) = h_\L(b),
        \mand
        h_\cS(b^{-1}) = h_\cT(b^{-1}).
    \]
    
    Using 
    \Cref{eq:hSbFinalBound} with $\K$ replaced by $\L$ and other parameters replaced 
    by the upper bounds in~\Cref{eq:NewParam-L} we conclude the proof of \Cref{eq:SPartBoundGeneric}.

    \subsubsection*{--- Proving \Cref{eq:SPartBound2}:} \quad
    This is the same proof as above, except using
    \Cref{eq:DecomposableBound2} in place of \Cref{eq:DecomposableBound1} in the derivation of
    \Cref{eq:hLHSRHSineq}.

    \section{Proofs of \texorpdfstring{\Cref{thm:MultDepNorthcottNF,thm:MultDepGeneralThm}}{Theorems~\ref{thm:MultDepNorthcottNF} and~\ref{thm:MultDepGeneralThm}}}

\subsection{Dynamical canonical height function}
We introduce the dynamical canonical height function
which is useful in the proofs of
\Cref{thm:MultDepNorthcottNF,thm:MultDepGeneralThm}.

The following  result is standard and proofs of its statements can be found in~\cite[Section~3.4]{Silv};
see also~\cite[Remark~B.2.7]{HindrySilverman} and~\cite[Proposition~3.2]{Zan}
regarding the effectiveness of the result.

\begin{lemma}
    \label{lem:canonicalheights}
    For a fixed $f \in \K(X)$ with $\degf = \deg f \geq 2$
    there exists a function $\cht_f: \K \to [0, \infty)$
    such that:
    \begin{enumerate}[(a)] 
\item There is an effectively computable constant $C_6(\K,f) $ such that             \[
                \abs{\cht_f(\alpha) - h(\alpha)} < C_6(\K,f), 
            \]
            for all $\alpha \in \K$.
        \item For all $\alpha \in \K$ we have
            \[
                \cht_f(f(\alpha)) = \degf\cht_f(\alpha).
            \]
        \item For all $\alpha \in \K$ we have
            \[
                \cht_f(\alpha) = 0 \iff \alpha \in \PrePer(f).
            \]
    \end{enumerate}
\end{lemma}

As a consequence, there exists an effectively computable constant $C_7(\K,f)$
such that for all $\ell \in \Z$, $\ell  \geq 0$ and $\alpha \in \K$, $h(\alpha) > C_7(\K,f)$,
\begin{equation}
    \label{eq:hfnalphaIneqs}
    \degf^\ell C_7(\K,f) h(\alpha) > h(f^{(\ell )}(\alpha)) > \degf^\ell  C_7(\K,f)^{-1}h(\alpha).
\end{equation}

\subsection{Proof of \texorpdfstring{\Cref{thm:MultDepNorthcottNF}}{Theorem~\ref{thm:MultDepNorthcottNF}}}
We first prove~\Cref{eq:MultDepNorthcottBound1}.
Suppose $\alpha$ satisfies \Cref{eq:mnuIntegers}.
    Let $\degf = \deg f$.
    We assume that
    \begin{equation}
        \label{eq:halpha large}
        h(\alpha) >
        \max\left\{
        \frac{\degf+1}{\degf-1} C_6(\K,f),
        \ 
        2\eta_1(\K,f,\cS)^{-1}
        \right\}
    \end{equation}
    with $C_6(\K,f)$ as in \Cref{lem:canonicalheights}~(a) and
    $\eta_1(\K,f,\cS)$ as in \Cref{thm:SPartBoundNF}.

    The first term in the maximum on the right hand side of~\Cref{eq:halpha large}, along with \Cref{lem:canonicalheights}~(a) and~(b), ensures that
    \begin{equation}
        \label{eq:hfalphabigger}
        \begin{split} 
        h(f(\alpha))&  >  \cht_f(f(\alpha))   - C_6(\K,f) = \degf \cht_f(\alpha)   - C_6(\K,f) \\
        &  >  \degf  h(\alpha) - (\degf+1)  C_6(\K,f) >  h(\alpha).
        \end{split}
    \end{equation}  
                        
    The second term in the maximum in~\Cref{eq:halpha large} and \Cref{thm:SPartBoundNF}, along with~\Cref{eq:hfalphabigger}, implies that
    \begin{equation}
        \label{eq:hMkSlowerbound}
        \begin{split}
            h_{\MK \setminus \cS}(f^{(l)}(\alpha)^{-1})
            &> \eta_1(\K,f,\cS) h(f^{(l)}(\alpha)) - 1 \\
            &> \frac{\eta_1(\K,f,\cS)}{2} h(f^{(l)}(\alpha)), 
        \end{split}
    \end{equation}
    for any iterate $f^{(l)}(\alpha)$ with $l \geq 1$.

    For any $\alpha \in \o$, we  write $\yideal_\cS(\alpha)$ to mean
    the $\cS$-free part of $[\alpha]$, that is, the ideal 
    \[
        \yideal_\cS(\alpha) = \frac{[\alpha]}{\prod_{\fp \in \cS_0} \fp^{\ord_{\fp}\, \alpha}}.
    \]
    We now write $[f^{(m)}(\alpha)] = \fa \cdot \fb$ where
    \[\fa = \yideal_\cS(f^{(m)}(\alpha)) \mand
    \fb = \prod_{\fp \in \cS_0} \fp^{\ord_{\fp}\, f^{(m)}(\alpha)}.
    \]
    Observe that \Cref{eq:mnuIntegers} implies that $\fa \mid f^{(n)}(\alpha)$.
Setting
\[
k = m-n > 0, 
\]
we write
\[
    f^{(k)}(f^{(n)}(\alpha)) = f^{(m)}(\alpha)
\]
which, with the above observation,
implies that $\fa \mid f^{(k)}(0)$.

Since $0$ is not a periodic point, we have $f^{(k)}(0) \ne 0$ and
combining the above observation
with the notation in \Cref{lem:canonicalheights}
we obtain
\begin{equation} 
\begin{split}
    \label{eq:hMkShfk0}
    h_{\MK \setminus \cS} (f^{(m)}(\alpha)^{-1})
    &\leq h_{\MK \setminus \cS} (f^{(k)}(0)^{-1})
    \\
    &\leq h(f^{(k)}(0)^{-1})
    = h(f^{(k)}(0))
   \\
    &< \degf^k \cht_f(0) + C_6(\K,f).
\end{split}
\end{equation}

On the other hand, \Cref{eq:hMkSlowerbound} along with \Cref{lem:canonicalheights}~(a) 
implies that
\begin{equation}
    \label{eq:hMkSfmaGThfma}
    h_{\MK \setminus \cS}(f^{(m)}(\alpha)^{-1})
    > \frac{\eta_1(\K,f,\cS)}{2} (\degf^m \cht_f(\alpha) - C_6(\K,f)).
\end{equation}

Comparing \Cref{eq:hMkShfk0} and \Cref{eq:hMkSfmaGThfma},
since $k \leq m$,
we obtain
\[
    \cht_f(\alpha) < 2 \eta_1(\K,f,\cS)^{-1} \cht_f(0)
    + \frac{2\eta_1(\K,f,\cS)^{-1} + 1}{\degf^m}C_6(\K,f).
\]
As $\cht_f(0) < C_6(\K, f)$
we obtain the upper bound
\[
    h(\alpha) < \left(2\eta_1(\K,f,\cS)^{-1} + 1\right)\left(1 + \frac{1}{\degf^m}\right)C_6(\K,f), 
\]
as required.

The proof for~\Cref{eq:MultDepNorthcottBound2} is the same, except with $\eta_2$ instead of $\eta_1$ throughout.

\subsection{Proof of \texorpdfstring{\Cref{thm:MultDepGeneralThm}}{Theorem~\ref{thm:MultDepGeneralThm}}}

First we establish the result when one of $r$ or $s$ is equal to 0.
We obtain an explicit dependence on $\cS$ which may be interesting in its own right.
In particular, this gives a somewhat explicit version of~\cite[Proposition~1.5(a)]{Kriegeretal}.

More generally, an explicit version of \Cref{lem:falphainSunits} for $f(z) \in \K(z)$ can be derived from the
proof of~\cite[Proposition~1.5(a)]{Kriegeretal}.
The key ingredient of the proof is Siegel's Theorem for curves of genus 0
which can be made fully explicit using Baker's method
(see, for example, the end of~\cite[Theorem~4.3]{BakerBook}).

\begin{lemma}
    \label{lem:falphainSunits}
    Let $f(X) \in \o[X]$ be a polynomial with at least 3 distinct roots.
    Suppose that $\alpha \in \o$ satisfies
    \[
        f(\alpha) \in \o_\cS^*.
    \]
    Then
    \[
        h(\alpha) < \frac{\eta_1(\K,f,\cS)^{-1}}{\degf}
        + \left(1 + \frac{1}{\degf}\right) C_6(\K, f), 
    \]
    with $\eta_1(\K, f, \cS)$ as in \Cref{thm:SPartBoundNF}, $\degf = \deg f$
    and $C_6(\K,f)$ as in \Cref{lem:canonicalheights}.
\end{lemma}
\begin{proof}
    If $f(\alpha) \in \o_\cS^*$, then $h_\cS(f(\alpha)^{-1}) = h(f(\alpha)^{-1}) = h(f(\alpha))$.
    Substituting into \Cref{thm:SPartBoundNF} we obtain
    \begin{equation}
        \label{eq:falphainSunitsProof1}
        h(f(\alpha)) < \eta_1(\K,f,\cS)^{-1}. 
    \end{equation}
    By \Cref{lem:canonicalheights} we have the inequality
    \begin{equation} 
\begin{split}
        \label{eq:falphachtbound}
        h(f(\alpha)) &> \cht_f(f(\alpha)) - C_6(\K,f) = \degf\cht_f(\alpha) - C_6(\K,f) \\
                     &> \degf h(\alpha) - (\degf+1) C_6(\K,f).
\end{split}
\end{equation}
    The result now follows from substituting \Cref{eq:falphachtbound} into \Cref{eq:falphainSunitsProof1}.
\end{proof}

We now prove \Cref{thm:MultDepGeneralThm}.
\Cref{thm:MultDepGeneralThm} essentially follows from the proof of~\cite[Theorem~1.7]{BOSS}
(in the case where $\alpha \in R_{\cS_{f,\Gamma}}$) 
upon replacing the use of~\cite[Theorem~1.2]{BOSS} with \Cref{lem:falphainSunits}
and the use of~\cite[Theorem~1.3]{BOSS} with \Cref{thm:MultDepNorthcottNF}.  
We use the same cases as in the proof of~\cite[Theorem~1.7]{BOSS} and just indicate the changes necessary.

    As in~\cite{BOSS}, we can effectively bound the height of elements of $\PrePer(f)$ (see \Cref{lem:canonicalheights}~(a) and~(c)),
    hence we assume $\alpha \notin \PrePer(f)$ from now on.

    First, if $r = 0$ or $s = 0$, we then have that $f^{(n)}(\alpha) \in \o_\cS^*$ for some $n \geq 1$.
    \Cref{lem:falphainSunits} bounds the height of $f^{(n-1)}(\alpha)$.
    From this, \Cref{lem:canonicalheights} provides an effective upper bound on  $h(\alpha)$ as required.

    By replacing $(r,s)$ by $(-r,-s)$ we may assume that $r > 0$.

    If, in addition, $s < 0$, then, as in~\cite{BOSS}, we can conclude that $f^{(m)}(\alpha) \in \o_\cS^*$ and bound $h(\alpha)$ as above.

    If either $s \geq 2$ or $r \geq 2$, then the argument in~\cite{BOSS}
    applies directly
    (noting that as $\deg f \geq 3$, we can always apply one of~\cite[Theorem~2.1]{BEGBerczes} or~\cite[Theorem~2.2]{BEGBerczes} to obtain effective results).

    Finally, the case $r = s = 1$ is just a consequence of \Cref{thm:MultDepNorthcottNF}, which concludes the proof.

\section{Proof of \texorpdfstring{\Cref{thm:LowerBoundm}}{Theorem~\ref{thm:LowerBoundm}}}

\subsection{The case where \texorpdfstring{$m$}{m} is much larger than \texorpdfstring{$n$}{n}}

\begin{lemma}
    \label{lem:MultDepPrimeSupportDependsOnmn}
    Let $f(X)\in\o[X]$ be a polynomial with at least 3 distinct roots.
    Let $\alpha \in \o$, $m, n \in \Z$, $m > n \geq 0$ such that $f^{(m)}(\alpha), f^{(n)}(\alpha) \neq 0$.
    Let
    \[
        L = 
        \log^* \left(\frac{h\left(f^{(m)}(\alpha)\right)}{h\left(f^{(n)}(\alpha)\right) + 1}\right).
    \]
    Then
    \[
        \largestsupp\left(\frac{f^{(m)}(\alpha)}{f^{(n)}(\alpha)}\right)
        >
        C_8(\K, f)
        \frac{L \log^* L}{\log^* \log^* L},
    \]
    where $C_8(\K, f) > 0$ is an effectively computable constant.
\end{lemma}
\begin{proof}
For any $X > 0$, let $\cS^X = \MK^{\infty} \cup \MK^{\leq X}$,
where
\[
    \MK^{\leq X} = \{\abs{\cdot}_{v_\mathbf{p}} \mid \Nm(\mathbf{p}) \leq X\}.
\]
If $X > 1$, then, since at most $d$ prime ideals lie over each prime $p \in \Z$,
using an explicit bound on the prime counting function~\cite[Theorem~4.6]{Apostol}
we derive
\begin{equation}
    \label{eq:sXtX}
    \begin{split}
    & t_X = \# \(\cS^X \cap \MK^0\) \leq 6dX/\log X, \\
    & s_X =  \# \cS^X \leq d + 6dX/\log X.
\end{split}
\end{equation}

Let
\[
    X = \largestsupp\left(f^{(m)}(\alpha) / f^{(n)}(\alpha)\right).
\]
Then $\left(f^{(n)}(\alpha) / f^{(m)}(\alpha)\right) \in \o_{\cS^X}$.
Therefore,
\begin{equation}
    \label{eq:PropProofhMSUpperBound}
\begin{split}
    h_{\MK \setminus \cS^X}(f^{(m)}(\alpha)^{-1})
    &\leq
    h_{\MK \setminus \cS^X}(f^{(n)}(\alpha)^{-1}) \\
    &\leq
    h(f^{(n)}(\alpha)^{-1}) = h(f^{(n)}(\alpha)).
\end{split}
\end{equation}
Suppose that $X > 1$, hence $t_X > 0$.
Then \Cref{thm:SPartBoundNF}
implies that
\begin{equation}
    \label{eq:PropProofhMSLowerBound}
    h_{\MK \setminus \cS^X}(f^{(m)}(\alpha)^{-1})
    > \eta_2(\K,f,\cS^X) \cdot h(f^{(m)}(\alpha)) - 1.
\end{equation}
Combining \Cref{eq:PropProofhMSUpperBound} and \Cref{eq:PropProofhMSLowerBound}
we obtain
\begin{equation}
\label{eq:HRatioBound1}
    \frac{h\left(f^{(m)}(\alpha)\right)}{h\left(f^{(n)}(\alpha)\right) + 1}
    <
    \eta_2(\K,f,\cS^X)^{-1}.
\end{equation}
We note that
\begin{equation}
    \label{eq:boundsWithXlogX}
\begin{split}
    & A_2(d \dD \fh_\L, t_X\dD) \leq C_9(\K, f)^{X/\log X}
\end{split}
\end{equation}
for some effectively computable constant $C_9(\K, f)$.

Substituting \Cref{eq:sXtX} and \Cref{eq:boundsWithXlogX}
into \Cref{eq:HRatioBound1}
we obtain
\[
    \frac{h\left(f^{(m)}(\alpha)\right)}{h\left(f^{(n)}(\alpha)\right) + 1}
    < \left(C_{10}(\K, f) \log^* X\right)^{C_{10}(\K, f) X/\log^* X},
\]
where $C_{10}(\K, f) > 0$ is effectively computable.
Taking logs, we obtain
\[
    \log^*\left(\frac{h\left(f^{(m)}(\alpha)\right)}{h\left(f^{(n)}(\alpha)\right) + 1}\right)
    < C_{11}(\K, f) X \frac{\log^* \log^* X}{\log^* X},
\]
where $C_{11}(\K, f) > 0$ is effectively computable.
The desired result follows after some simple calculation.

In the case where $X = 1$, the same procedure but using
\Cref{eq:SPartBoundGeneric} instead of \Cref{eq:SPartBound2}
shows that
\[
    \frac{h\left(f^{(m)}(\alpha)\right)}{h\left(f^{(n)}(\alpha)\right) + 1}
    <
    C_{12}(\K, f),
\]
where $C_{12}(\K,f)$ is an effectively computable constant, as required.
\end{proof}

\subsection{The case where \texorpdfstring{$m$}{m} and \texorpdfstring{$n$}{n} are of comparable sizes} 
\begin{lemma}
    \label{lem:MultDepPrimeSupport}
    Let $f(X)\in\o[X]$ be a polynomial with at least 3 distinct roots
    and for which $0$ is not periodic.
    Let $\alpha \in \o$, $m, n \in \Z$, $m > n \geq 0$ such that $f^{(m)}(\alpha)$, $f^{(n)}(\alpha) \neq 0$.
    Then
    \[
        \largestsupp\left(\frac{f^{(m)}(\alpha)}{f^{(n)}(\alpha)}\right)
        >
        C_{13}(\K, f)
        \frac{\loght\left(f^{(n)}(\alpha)\right) \log^* \loght\left(f^{(n)}(\alpha)\right)}{\log^* \log^* \loght\left(f^{(n)}(\alpha)\right)},
    \]
    where $C_{13}(\K,f) > 0$ is an effectively computable constant.
\end{lemma}
\begin{proof}

Define $\cS^X$ and $X$ as in the proof of \Cref{lem:MultDepPrimeSupportDependsOnmn}.

Then $\left(f^{(n)}(\alpha) / f^{(m)}(\alpha)\right) \in \o_{\cS^X}$.

If $X = 1$, hence $\cS^X = \MK^{\infty}$,
applying~\Cref{thm:MultDepNorthcottNF} with $\cS = \MK^{\infty}$
and $\alpha = f^{(n)}(\alpha)$
yields an effective upper bound on $h(f^{(n)}(\alpha))$
in terms of $\K$ and $f$, as required.

Otherwise, $\# (\cS^X \cap \MK^{0}) > 0$.
Hence, applying~\Cref{thm:MultDepNorthcottNF}
with $\cS = \cS^X$ and $\alpha = f^{(n)}(\alpha)$, we obtain
\begin{equation}
    \label{eq:hBoundWithSx}
    \begin{split}
        \frac{h(f^{(n)}(\alpha))}{c_2(\K,f)} <\ &
    \eta_2(\K,f,\cS^X)^{-1}.
\end{split}
\end{equation}
We can now proceed as in the proof of \Cref{lem:MultDepPrimeSupportDependsOnmn},
except with \Cref{eq:hBoundWithSx} in place of \Cref{eq:HRatioBound1},
to obtain the desired result.
\end{proof}

\subsection{Concluding the proof} 
We now prove \Cref{thm:LowerBoundm}.

If $h(f^{(n)}(\alpha)) > C_7(\K,f)$,
with $C_7(\K, f)$ as in \Cref{eq:hfnalphaIneqs},
then
a combination of \Cref{lem:MultDepPrimeSupportDependsOnmn},
used for 
\[
    m-n \geq \frac{ \lambda(f^{(n)}(\alpha))}{\log \degf},
\] 
and
of \Cref{lem:MultDepPrimeSupport} otherwise
implies the result.

Otherwise, $h(f^{(n)}(\alpha)) \leq C_7(\K,f)$.
The result now follows from \Cref{lem:MultDepPrimeSupportDependsOnmn}.

\section{Proof of \texorpdfstring{\Cref{thm:WeakZsigmondy}}{Theorem~\ref{thm:WeakZsigmondy}}}

If $\mathbf{p} \mid f^{(m)}(\alpha)$ and $\mathbf{p} \mid f^{(n)}(\alpha)$
with $m > n$,
then, writing
\[
    f^{(m)}(\alpha) = f^{(m-n)}\left(f^{(n)}(\alpha)\right),
\]
we see that $\mathbf{p} \mid f^{(m-n)}(0)$.

Fix a $k \in \Z$, $k > 0$ and
let $\cS_k$ be the finite set of places containing $\MK^\infty$
and all finite places corresponding to a prime dividing a value in the set
\[
    \{f^{(1)}(0), f^{(2)}(0), \cdots, f^{(k)}(0)\}
\]
($\cS_k$ is finite since $0$ is not periodic).

With the above observation, to prove the desired statement for
\[
    k(m, \alpha) = k,
\]
it suffices to show that
\[
    h_{\cS_k}(f^{(m)}(\alpha)^{-1}) < h(f^{(m)}(\alpha)).
\]

The case where $\cS_k$ contains no finite places is trivial.
Hence, we assume that $\cS_k$ contains at least one finite place.
Then, by \Cref{thm:SPartBoundNF}, we have that
\begin{equation}
    \label{eq:LowerBoundhdiff}
    h(f^{(m)}(\alpha)) - h_{\cS_k}(f^{(m)}(\alpha)^{-1}) > \eta_2(\K, f, \cS_k)h(f^{(m)}(\alpha)) - 1.
\end{equation}

We note the following inequalities for $\cS_k$ which are consequences of \Cref{lem:canonicalheights} and simple calculation
(for the last inequality note that
$\sum_{i=1}^t \log (N(\mathbf{p_i})) < C_{14}(\K, f) \degf^k$):
\[t <  C_{14}(\K, f) \degf^k, \qquad P < e^{  C_{14}(\K, f) \degf^k},  \qquad 
  \prod_{i=1}^t \log (N(\mathbf{p_i})) < e^{C_{14}(\K, f) \degf^k}
\]
for an effectively computable $C_{14}(\K, f)$. Hence,
\begin{equation}
    \label{eq:Proofeta2cSkBound}
    \eta_2(\K, f, \cS_k)^{-1} < e^{C_{15}(\K, f) \degf^k}.
\end{equation}
Substituting \Cref{eq:Proofeta2cSkBound}
into \Cref{eq:LowerBoundhdiff}, the required statement holds for any $k$ such that
\[
    \log h(f^{(m)}(\alpha)) >  C_{15}(\K, f) \degf^k.
\]

If $h(f^{(m)}(\alpha))$ is sufficiently large, then \Cref{thm:WeakZsigmondy} follows immediately.
Otherwise, $h(f^{(m)}(\alpha))$ is bounded and we may pick $c_6(\K, f)$
small enough such that $k(m, \alpha) = 0$.

\section*{Acknowledgement}

The authors are grateful to Attila B\' erczes for supplying a proof of a version of 
\Cref{thm:SPartBoundNF} in terms of the norm of the $\cS$-part of 
$f(\alpha)$ and Alina Ostafe for her encouragement and comments on an initial draft of the paper.

This work  was  supported, in part,  by the Australian Research Council Grant~DP180100201.


\begin{thebibliography}{99}

    \bibitem{Apostol} T. M. Apostol, \textit{Introduction to analytic number theory},
Springer-Verlag, Berlin, Heidelberg, 1976.

\bibitem{BakerBook} A. Baker, \textit{Transcendental number theory},
Cambridge Univ. Press, 1975. 

\bibitem{Bar1}
F. Barroero, `Counting algebraic integers of fixed degree and bounded height', 
 \textit{Monatsh.\ Math.\/}, {\bf 175} (2014), 25--41.

\bibitem{Bar2}
F. Barroero, `Algebraic $\cS$-integers of fixed degree and bounded height', 
 \textit{Acta Arith.\/}, {\bf 167} (2015), 67--90.

\bibitem{BEGBerczes} A. B\' erczes, J.-H. Evertse and K.  Gy{\"o}ry, `Effective results for hyper- and superelliptic equations over number fields', \textit{Publ. Math. Debrecen}, {\bf 82} (2013), 727--756.

  \bibitem{BOSS} A. B\' erczes, A. Ostafe, I. E. Shparlinski and J. H. Silverman, `Multiplicative dependence among iterated values of rational functions modulo finitely generated groups',   \textit{Internat. Math. Res. Notices}, (to appear). 

\bibitem{BomGub} 
E. Bombieri and W. Gubler,  \textit{Heights in Diophantine geometry\/},   Cambridge Univ. Press, 2006.

\bibitem{BEG} Y. Bugeaud, J.-H. Evertse and K.  Gy{\"o}ry, 
`$\cS$-parts of values of univariate polynomials, binary forms and decomposable forms at integral points',
  \textit{Acta  Arith.\/}, {\bf 184} (2018),  151--185.



\bibitem{BugeaudGyory} Y. Bugeaud and K.  Gy{\"o}ry,
`Bounds for the solutions of unit equations',
 \textit{Acta Arith.\/}, {\bf 74} (1996), 67--80.
 
 \bibitem{Falt1} G. Faltings, `Endlichkeitss\" atze fur abelsche
  Variet\" aten \" uber Zahlkorpern', \textit{Invent. Math.}, {\bf 73}
  (1983), 349--366.

\bibitem{Falt2} G. Faltings, `Finiteness theorems for abelian
  varieties over number fields', \textit{Arithmetic geometry, Storrs,
    Connecticut, 1984}, Springer, New York, 1986.

\bibitem{Gyory} K. Gy{\"o}ry,
`Bounds for the solutions of $S$-unit equations and decomposable form equations II',  \textit{Preprint}, 2019, 
available at \url{https://arxiv.org/abs/1901.11289}.

\bibitem{GyoryYu} K. Gy{\"o}ry and K. Yu,
`Bounds for the solutions of $S$-unit equations and decomposable form equations',
 \textit{Acta  Arith.\/}, {\bf 123} (2006),  9--41.
 
\bibitem{HindrySilverman}
M. Hindry and J. H. Silverman, \textit{Diophantine geometry: An introduction},
Springer-Verlag, New York, 2000.

\bibitem{HsiaSilverman} L.-C. Hsia and J. H. Silverman,
`A quantitative estimate for quasi-integral points in orbits',
 \textit{Pacific J. Math\/}, {\bf 249} (2011), 321--342.

\bibitem{Kriegeretal} H. Krieger, A. Levin, Z. Scherr, T. Tucker, Y. Yasufuku and M. E. Zieve,
`Uniform boundedness of $\cS$-units in arithmetic dynamics',
 \textit{Pacific J. Math\/}, {\bf 274} (2015), 97--106.

\bibitem{Neukirch} J. Neukirch,  \textit{Algebraic Number Theory\/},
Springer-Verlag, Berlin, Heidelberg, 1999.

\bibitem{North}
D. G. Northcott, `Periodic points on an algebraic variety',
 \textit{Ann. of Math.\/}, \textbf{51} (1950), 167--177.

\bibitem{OSSZ}
A. Ostafe, M. Sha, I. E. Shparlinski and U. Zannier, `On multiplicative dependence of values of rational functions and a generalisation 
of the Northcott theorem', 
 \textit{Michigan Math. J.\/}, {\bf 68} 2019, 385--407. 
 
 
 \bibitem{OstShp}
A. Ostafe, L. Pottmeyer and I. E. Shparlinski, `Perfect powers in value sets and orbits of polynomials', 
 \textit{Preprint\/},   2019, available at \url{https://arxiv.org/abs/1907.12057}.

\bibitem{Silv}
J. H. Silverman,  \textit{The arithmetic of dynamical systems},
Springer-Verlag, New York, 2007.

\bibitem{Voutier} P. M. Voutier,
`An effective lower bound for the height of algebraic numbers',  \textit{Acta Arith.\/}, {\bf 74} (1996), 81--95.

\bibitem{Wid1}
M.~Widmer, `Counting points of fixed degree and bounded height', 
 \textit{Acta Arith.\/}, {\bf 140} (2009), 145--168.

\bibitem{Wid2}
M.~Widmer, `Integral points of fixed degree and bounded height',  \textit{Int. Math. Res. 
Notices\/}, {\bf  2016} (2016), 3906--3943.

\bibitem{Zan} U.  Zannier, 
 \textit{Lecture notes on Diophantine analysis\/},
Publ. Scuola Normale Superiore, Pisa, 2009. 

\end{thebibliography}
\end{document}